\documentclass[a4paper,
              12pt,
              twoside
              ]
              {article}

\usepackage{latexsym, comment}
\usepackage{amsmath,amsthm,mathtools,amssymb,eqnarray}

\usepackage[utf8]{inputenc}
\usepackage[english]{babel}
\usepackage{csquotes}

\newcommand{\new}[1]{{{\color{red}\bf#1}}}

\usepackage[
dashed=false,
natbib=true,
sorting=nyt,%
style=ieee,%
language=british,%
citestyle=numeric-comp,%
backend=bibtex,% bibtex or biber
backref=true,%
hyperref=true,%
maxcitenames=8,%
maxbibnames=100,%
block=none,%
url=false,
isbn=false,
doi=false,
defernumbers=true,
sortcites=false]{biblatex}
\DefineBibliographyStrings{english}{%
  backrefpage  = {cited on page}, % for single page number
  backrefpages = {cited on pages} % for multiple page numbers
}

\usepackage{bookmark}
\bookmarksetup{
  numbered,
  open,
}

\usepackage{geometry}
\usepackage{fancyhdr}

\usepackage{enumerate} % e.g (i)

\usepackage{booktabs}
\usepackage{xcolor}

\usepackage[normalem]{ulem}

\title{On large Sidon sets}
\author{
  Ingo Czerwinski\thanks{Otto-von-Guericke-University Magdeburg,
        Faculty of Mathematics, Institute for Algebra and Geometry,
        39106 Magdeburg, Germany,
        (\texttt{ingo@czerwinski.eu}, \texttt{alexander.pott@ovgu.de})} \!
        \thanks{Corresponding author} \,
  and Alexander Pott\footnotemark[1]
}
\date{}

\newcommand{\moreSpaceBetweenRows}[1]{\renewcommand{\arraystretch}{#1}}

\newtheorem{thm}{Theorem}[section]
\newtheorem{cor}[thm]{Corollary}

\newtheorem{prop}[thm]{Proposition}
\newtheorem{defn}[thm]{Definition}
\newtheorem{rem}[thm]{Remark}

\newcommand{\secref}[1]{Section~\ref{#1}}

\newcommand{\thmref}[1]{Theorem~\ref{#1}}
\newcommand{\corref}[1]{Corollary~\ref{#1}}

\newcommand{\propref}[1]{Proposition~\ref{#1}}

\newcommand{\remref}[1]{Remark~\ref{#1}}

\newcommand{\tabref}[1]{Table~\ref{#1}}

\newcommand{\inte}{\mathbb{Z}}

\newcommand{\binf}{\mathbb{F}_2}
\newcommand{\binfn}[1]{\mathbb{F}_{2^{#1}}}
\newcommand{\binvn}[1]{\mathbb{F}_2^{#1}}

\newcommand{\eps}{\varepsilon}

\providecommand{\transpose}[1]{#1^\intercal}

\newcommand{\set}[1]{\lbrace #1 \rbrace}
\newcommand{\sett}[2]{\lbrace #1 : #2 \rbrace}
\newcommand{\setminzero}{\setminus\set{0}}

\providecommand{\vspan}[1]{\langle #1\rangle}

\providecommand{\lin}[1]{\mathcal{L}(#1)}

\providecommand{\wtrf}[1]{W_{#1}}

\providecommand{\ftrf}[1]{\mathcal{F}_{#1}}

\providecommand{\code}{\mathcal{C}}
\providecommand{\checkMat}{\mathcal{H}}

\providecommand{\smax}[1]{s_{\max{}}(\binvn{#1})}

\providecommand{\abs}[1]{\left\lvert#1\right\rvert}

\providecommand{\floor}[1]{\left\lfloor#1\right\rfloor}

\providecommand{\highlight}[1]{\textbf{#1}}
\providecommand{\mathKeyword}[1]{\emph{#1}}

\providecommand{\keywords}[1]{\textbf{Keywords} #1 \\[12px]}
\providecommand{\subclass}[1]{\textbf{Mathematics Subject Classification (2020)} #1 }

\begin{document}

\maketitle

\renewcommand{\sectionmark}[1]{}

\renewcommand{\labelenumi}{(\alph{enumi})} % changed enum to (a), (b), ...

\begin{abstract}
A (binary) Sidon set $M$ is a subset of $\mathbb{F}_2^t$
such that the sum of four distinct elements of $M$ is never 0.
The goal is to find Sidon sets of large size.
In this note we show that the graphs of
almost perfect nonlinear (APN) functions with high linearity
can be used to construct large Sidon sets.
Thanks to recently constructed APN functions $\mathbb{F}_2^8\to \mathbb{F}_2^8$
with high linearity, we can construct Sidon sets of size 192 in $\mathbb{F}_2^{15}$,
where the largest sets so far had size 152.
Using the inverse and the Dobbertin function also gives larger Sidon sets as previously known.
Each of the new large Sidon sets $M$ in $\mathbb{F}_2^t$ yields a binary linear code
with $t$ check bits, minimum distance 5, and a length not known so far.
Moreover, we improve the upper bound for the linearity of arbitrary APN functions.
\end{abstract}


\subclass{Primary: 11B13, 94D10.\\ Secondary: 94B05.}

% 94-XX Information and communication theory, circuits
% 94Axx Communication, information
% 94A60 Cryptography

% 94-XX Information and communication theory, circuits
% 94Bxx Theory of error-correcting codes and error-detecting codes
% 94B05	Linear codes (general theory)
% 94B65 Bounds on codes
% 94D10 Boolean functions [See also 06E30] {For connections with circuits and networks, see 94C11}

% 51-XX Geometry {For algebraic geometry, see 14-XX; for differential geometry, see 53-XX}
% 51Exx Finite geometry and special incidence structures
% 51E15 Finite affine and projective planes (geometric aspects)

% 11-XX Number theory
% 11Bxx Sequences and sets
% 11B13 Additive bases, including sumsets [See also 05B10]

% 05-XX Combinatorics {For finite fields, see 11Txx}
% 05Bxx Designs and configurations {For applications of design theory, see 94C30}
% 05B25 Combinatorial aspects of finite geometries [See also 51D20, 51Exx]

\section{Introduction}

The following question is easy to state but seems difficult to answer:
What is the largest number of elements
in a set, in which an appropriate addition is defined,
so that no sum appears twice?

In the early 1930s Sidon \cite{Sidon1932, Sidon1935} was the first who discussed
such sets in the integers.
Later Babai and S{\'o}s \cite{BabaiSos1985} generalised
such sets to arbitrary groups and called them
\highlight{Sidon sets}. In this note, we only discuss binary Sidon sets, that are
Sidon sets in $\binvn{t}$, the $t$-dimensional vector space over the binary field $\binf$:

\begin{defn}[\cite{BabaiSos1985}]
  A subset $M$ of $\binvn{t}$  is called a \mathKeyword{Sidon} set
  if $m_1+m_2\ne m_3+m_4$ for all distinct $m_1,m_2,m_3,m_4\in M$.
  Equivalently, the sum of any four distinct elements of $M$ is different from $0$.
\end{defn}

We emphasise that, by definition, every subset of a Sidon set
$M$ and every coset $M+b= \sett{m+b}{m\in M}$, $b\in\binvn{t}$,
is also a Sidon set.

The main research question on Sidon sets in $\binvn{t}$ is
to find sets of large size.
For dimensions $t\geq 11$, the maximum size is unknown (see \tabref{tab:SidonSizesOverview}).
Consequently, it is interesting to improve lower bounds by constructions,
as well as upper bounds by theoretical arguments.

The trivial upper bound (already mentioned in \cite{BabaiSos1985})
arises from the fact that the sums of two distinct elements
of a Sidon set $M$ in $\binvn{t}$ are distinct and non-zero,
hence
\[
\binom{\abs{M}}{2} = \frac{\abs{M}(\abs{M}-1)}{2}
\leq \abs{\binvn{t}\setminzero}.
\]

So far, progress on the upper bound has only been made in terms of coding theory.
This uses a one-to-one correspondence between %sum-free
Sidon sets that contain zero and linear codes with minimum distance $\geq 5$
(\cite[p. 33]{MacWilliamsSloane1978}, \cite{CohenZemor1999}, \cite{CzerwinskiPott2024}).

The first improvement of the trivial bound was made in 1993 by
Brouwer and Tolhuizen \cite{BrouwerTolhuizen1993}. They improved the case $t$ odd.
The case $t$ even was recently improved by the authors
in \cite{CzerwinskiPott2024}. We state here only the Brouwer-Tolhuizen version.

\begin{prop} \label{prop:max-sidon-upper-bound}
  For any $t\geq 6$, an upper bound for the maximum size of a Sidon set
  in $\binvn{t}$ is given by
  \begin{equation} % equRefTrivialBound
  \smax{t} \leq
  \begin{cases}
    \sqrt{2^{t+1}} - 2            &\text{ for $t$ odd},\\
    \floor{\sqrt{2^{t+1}} + 0.5}  &\text{ for $t$ even}.
  \end{cases}
  \end{equation}
\end{prop}

The improvement in \cite{CzerwinskiPott2024} is that we may subtract
$1$ or $2$ in the case that $t$ is even, depending on some technical conditions.

For small dimensions $t\leq 14$, better upper bounds are known derived from coding theory
(see e.g Table 2 of \cite{CzerwinskiPott2024} or \tabref{tab:SidonSizesOverview} in this paper).

On the construction side, the following is known.
In the case of dimension $t=2n$ even, the best
known infinite families are Sidon sets of size
$2^n+1$ if $n$ is odd and $2^n+2$ if $n$ is even.
Sidon sets of size $2^n+1$ are obtained from non-primitive BCH codes
which are constructed as shortenings of codes with parameters
$[2^n+1, 2^n+1-2n,5]$ (see \cite[p. 586]{MacWilliamsSloane1978} or (C2) construction in \cite{Nagy2025}).
Another construction is by Carlet, Mesnager and Picek
 \cite{CarletMesnager2022, CarletPicek2023}.
They have shown that the multiplicative subgroup $M$ of size $2^n+1$
in the finite field $\binfn{2n}$ is a Sidon set
in the additive group of the field.
This set can also be constructed from a certain Goppa code,
see \cite{Nagy2025}.
If $n$ is even, then $0\notin M$ and $M\cup \{0\}$ is Sidon too.
These Sidon sets are of maximum size for $t=4,6,8$.
For $t=10$ the construction leads to $33$
but the largest possible Sidon set in dimension $10$ has size $34$,
see \tabref{tab:SidonSizesOverview}.

In the case of odd dimension $t=2n-1$, less is known on infinite families of Sidon sets.
There is only one construction which uses the shortening of a BCH code
(see \cite[p. 586]{MacWilliamsSloane1978} or
construction (C3) in \cite{Nagy2025}).
It leads to Sidon sets in $\binvn{2n-1}$
of size $2^{n-1} + 2^{n/2}$ for $n$ even and
of size $2^{n-1} + 2^{\frac{n-1}{2}}$ for $n$ odd.
For $t=7$, this construction gives a Sidon set of maximum possible size.
For $t>7$ and $t$ odd, several sporadic examples are known which give
larger Sidon sets, see, for instance, \tabref{tab:SidonSizesOverview}.
To summarise:

\begin{rem} \label{rem:classical}
  The best infinite families of Sidon sets in $\binvn{t}$ for $t\geq 6$
  known so far lead to Sidon sets of size
  $$
  \begin{cases}
    2^n + 2                     &\text{ for $t=2n$ and $n$ even},\\
    2^n + 1                     &\text{ for $t=2n$ and $n$ odd},\\
    2^{n-1} + 2^{n/2}           &\text{ for $t=2n-1$ and $n$ even},\\
    2^{n-1} + 2^{\frac{n-1}{2}} &\text{ for $t=2n-1$ and $n$ odd}.\\
  \end{cases}
  $$
  We call these the \highlight{classical} parameters of large Sidon sets.
\end{rem}

Another important class of Sidon sets in even dimensions $t=2n$
are Sidon sets of size $2^n$ and
stems from almost perfect nonlinear functions.
These functions are quite important from a cryptographic perspective.

\begin{defn}[\cite{NybergKnudsen93}]
	A function $F\colon \binvn{n} \to \binvn{n}$ is called \mathKeyword{almost perfect nonlinear} (APN)
	if for every $a\in\binvn{n}\setminzero$ and $b\in\binvn{n}$ the equation
	$
	F(x) + F(x+a) = b
	$
	has at most $2$ solutions in $\binvn{n}$.
\end{defn}

Several classes of APN functions are known: there are some infinite families,
see \cite{CarletBook2020} for a summary, but also many sporadic constructions
(which are quite important for us), see
\cite{BeierleLeander2020}.

As early as 1969, Lindström used the graph of the cube function $x^3$,
which is APN, as an example of a Sidon set \cite{Lindstrom69}.
The relationship between APN functions and Sidon sets
is as follows (see e.g \cite{CarletMesnager2022, CarletPicek2023}):
\begin{prop} \label{prop:APNGraphSidon}
    A function $F\colon \binvn{n} \to \binvn{n}$ is APN
	if and only if its graph $G_F=\sett{(x,F(x))}{x\in\binvn{n}}$
%\subseteq\binvn{n}\times\binvn{n} $
	is a Sidon set in $\binvn{2n}$.
\end{prop}

Note that the graph of an APN function is in $\binvn{2n}$ and
has size $2^n$ which is slightly worse
than the other constructions mentioned above.
But they will play an important role in \secref{sec:NewConstruction}
to construct large Sidon sets in odd dimensions.

Now, that we have seen constructions of large Sidon sets,
we will introduce a new construction of large Sidon sets in the next section.
In its general form, an existing Sidon set is intersected with subspaces.
In \secref{sec:Linearity} we will then show that the linearity of an APN function
gives us information about the largest intersection of its graph with hyperplanes.
In addition, we improve the upper bound for the linearity of an arbitrary APN function.
In \secref{sec:NewConstruction} we then explicitly construct large Sidon sets
from APN functions with high linearity,
and in \secref{sec:linear-codes} we translate them into coding theory terms.
Finally, in \secref{sec:summary}, we summarise all results and give an overview table
about the maximum size of Sidon sets in small dimensions and related bounds and constructions.

\section{Sidon sets in subspaces}
\label{sec:subspaces}

The idea of constructing large Sidon sets
is based on the fact that every subset of a Sidon set is also
a Sidon set.
If $M$ is a Sidon set in $\binvn{t}$ and $U$ an $s$-dimensional
linear subspace of $\binvn{t}$, then
$M\cap U$ is a Sidon subset of $\binvn{t}$
and therefore a Sidon subset in the $s$-dimensional vector space $U$.

Also the intersection with affine subspaces can be used to construct Sidon sets,
since with $M$ also $M+b$ is Sidon for $b\in\binvn{t}$.
Hence, the intersection of $M$ with an
affine subspace $U+b$ leads to Sidon sets in $U$
by intersecting $M+b$ with $U$ (note $|M\cap (U+b)|=|(M+b)\cap U|$).

Our goal is to construct large Sidon sets using this procedure.
In order to get a large Sidon set in this way, you need a subspace $U$
which contains ``many'' elements of the Sidon set $M$.
But one expects that large Sidon sets are ``equally'' distributed
within $\binvn{t}$, hence $\abs{M\cap U}\approx |M|/2^{t-s}$ is the expectation.
In other words: if we intersect a large Sidon set with a \highlight{hyperplane},
we expect that the size of the intersection is in the range of $|M|/2$.
However, the upper bound reduces the size only by a factor $1/\sqrt{2}$ and
we present examples where the intersection size of $M$ with a hyperplane is in the range $|M|/\sqrt{2}$ (see \secref{sec:NewConstruction}).

Claude Carlet pointed out that the construction in Theorem 4.1 of
\cite{CarletPicek2023} may be also reformulated as a construction
of a large Sidon set as the intersection of the graph of an APN function
with an appropriate subgroup (not a hyperplane).

\section{Linearity and hyperplanes}
\label{sec:Linearity}

Throughout the paper $\cdot$ denotes the usual inner product in $\binvn{t}$,
which is $a\cdot b = a_1 b_1 + \dots + a_t b_t$ for
$a = (a_1, \dots, a_t), b = (b_1, \dots, b_t) \in\binvn{t}$.

Let $F\colon \binvn{n} \to \binvn{m}$ be a function.
The \mathKeyword{Walsh transform} of $F$ is the function
$\wtrf{F}\colon \binvn{n} \times \binvn{m} \to \inte$ defined by
\[
  \wtrf{F}(a,b) = \sum_{x\in \binvn{n}} (-1)^{a\cdot x + b\cdot F(x)}.
\]
The \mathKeyword{linearity} of $F$ is defined by
\[
  \lin{F} = \max_{
    a\in\binvn{n},b\in\binvn{m}\setminzero
  }\abs{\wtrf{F}(a,b)}.
\]

In fact, the Walsh transform and linearity of $F$ can also be formulated
as properties of the graph $G_F$ inside its ambient space $\binvn{n+m}$
(see \cite[p. 71]{CarletBook2020}):

\begin{defn} \label{def:FourierLin}
  Let $M$ be a subset of $\binvn{t}$. The {Fourier transform}
  of $M$ is the mapping $\ftrf{M} \colon \binvn{t} \to \inte$
  defined by
  \[
    \ftrf{M}(a) = \sum_{m\in M} (-1)^{a\cdot m}
  \]
  and the linearity of $M$ is defined by
  \[
    \lin{M} = \max_{a\in\binvn{t}\setminzero}\abs{\ftrf{M}(a)}.
  \]
\end{defn}

We remark that $\ftrf{M}(a)$ and $\lin{M}$ are even if and only if $\abs{M}$ is even,
and that the linearity of a function equals the linearity of its graph
(see \cite[p. 71]{CarletBook2020}):
%the linearity of the graph $G_F$ of a function (as in \defnref{def:FourierLin})
%is the same as the linearity of the function $F$.

\begin{prop} \label{prop:FourierEqualsWalsh}
If $F\colon \binvn{n} \to \binvn{m}$ is a function and $G_F$ its graph
then $\lin{F} = \lin{G_F}$.
%In other words, the linearity of the function equals the linearity of its graph.
\end{prop}
\begin{proof}
We have
\begin{align*}
\lin{F} &= \max_{
            a\in\binvn{n},b\in\binvn{m}\setminzero
            }\abs{W_F(a,b)} \\
    &= \max_{
            a\in\binvn{n},b\in\binvn{m}\setminzero
            }\abs{\sum_{x\in \binvn{n}} (-1)^{a\cdot x + b\cdot F(x)}}\\
    &= \max_{
            (a,b)\in\binvn{n}\times(\binvn{m}\setminzero)
            }\abs{\sum_{x\in \binvn{n}} (-1)^{(a,b)\cdot (x,F(x))}}
            \\
    &= \max_{
            (a,b)\in\binvn{n}\times\binvn{m}\setminus{(0,0)}
            }\abs{\sum_{x\in \binvn{n}} (-1)^{(a,b)\cdot (x,F(x))}}
    = \lin{G_F}.
\end{align*}
The last equality holds because of
$\sum_{x\in \binvn{n}} (-1)^{a\cdot x}=0$
and therefore
\[\sum_{x\in \binvn{n}} (-1)^{(a,0)\cdot (x,F(x))}=0\]
for $a\in\binvn{n}\setminzero$.
\end{proof}

The following Proposition is well known
using the terminology of functions (see \cite{ChabaudVaudenay1995}).
The interpretation in the terminology of sets seems to be folklore.
We include it here for the convenience of the reader.

\begin{prop} \label{prop:WalshLinConnection}
  Let $M$ be a subset of $\binvn{t}$,
  $a\in \binvn{t}\setminzero$,
  $H_a = \sett{g\in\binvn{t}}{a\cdot g = 0}$
  and let $b\in \binvn{t}\setminus H_a$.
  Then $\ftrf{M}(a)$ has different representations using intersection numbers with affine hyperplanes:
  \begin{align*}
    \ftrf{M}(a) &= \abs{M\cap H_a} - \abs{(M+b)\cap H_a} \\
                &= \abs{M} - 2\cdot\abs{(M+b)\cap H_a} \\
                &= 2\cdot\abs{M\cap H_a} - \abs{M}
  \end{align*}
  and
  \begin{align*}
    &\abs{M\cap H_a} \qquad   = \frac{\abs{M} + \ftrf{M}(a)}{2} \\
    &\abs{(M+b)\cap H_a} = \frac{\abs{M} - \ftrf{M}(a)}{2}.
  \end{align*}
\end{prop}
\begin{proof}
  \begin{align*}
    \ftrf{M}(a) &= \sum_{m\in M} (-1)^{a\cdot m} \\
                &= \sum_{m\in M\cap H_a} 1 \quad + \sum_{m\in (M+b)\cap H_a} -1 \\
                &= \abs{M\cap H_a} - \abs{(M+b)\cap H_a}
  \end{align*}
  The remaining follows from $\abs{M} = \abs{M\cap H_a} + \abs{(M+b)\cap H_a}$.
\end{proof}

\begin{cor} \label{cor:HyperplaneLin}
  If $M$ is a subset of $\binvn{t}$ then there exists a linear hyperplane $H$ such that either
  \[
    \abs{M\cap H} = \frac{\abs{M} + \lin{M}}{2} \qquad\qquad\qquad\qquad\quad
  \]
  or
  \[
    \abs{(M+b)\cap H} = \frac{\abs{M} + \lin{M}}{2} \qquad
      \text{ for every } b\in\binvn{t}\setminus H.
  \]
\end{cor}
\begin{proof}
  From the definition of $\lin{M}$ it follows
  that $\lin{M}=\abs{\ftrf{M}(a)}$ for some $a\in\binvn{t}\setminzero$.
  Let $H_a = \sett{g\in\binvn{t}}{a\cdot g = 0}$ and $b\in\binvn{t}\setminus H_a$ as above.
  \propref{prop:WalshLinConnection} shows the follwing:
  $\abs{M\cap H_a}  = \frac{\abs{M} + \ftrf{M}(a)}{2} = \frac{\abs{M} + \lin{M}}{2}$ if $\ftrf{M}(a)>0$,
  and
  $\abs{(M+b)\cap H_a} = \frac{\abs{M} - \ftrf{M}(a)}{2} = \frac{\abs{M} + \lin{M}}{2}$ otherwise.
\end{proof}

\begin{cor} \label{cor:HyperplaneSidonFromLin}
  Let $M$ be a  Sidon set in $\binvn{t}$.
  Then there exists a Sidon set in $\binvn{t-1}$
  of size $\frac{\abs{M} + \lin{M}}{2}$.
\end{cor}
\begin{proof}
    Use \corref{cor:HyperplaneLin} and note that $M+b$
    is Sidon too.
\end{proof}

\begin{cor}\label{cor:main}
  Let $F:\binvn{n}\to\binvn{n}$ be an APN function.
  Then there exists a Sidon set in $\binvn{2n-1}$ of size
  $2^{n-1}+\frac{\lin{F}}{2}$.
\end{cor}
\begin{proof}
    Direct consequence of \propref{prop:APNGraphSidon},
    \propref{prop:FourierEqualsWalsh},
    and \corref{cor:HyperplaneSidonFromLin}.
\end{proof}
In the next section, we will apply \corref{cor:main} to several highly linear APN functions.
But first, as an interesting application,
we use it to give an upper bound on the linearity of APN functions.

Not so much is known on that topic (see \cite{Carlet2018})
apart of that an arbitrary APN function $F\colon \binvn{n} \to \binvn{n}$
cannot have linearity $2^n$ for $n\geq 3$ (\cite[p. 424]{CarletBook2010}).
For certain classes of APN functions smaller bounds are known.
In \cite{Carlet2021Gamma} a bound is given,
that holds for all \textbf{known} APN functions.
But it is open if it is valid for \textbf{all} APN functions.
If $F$ is a quadratic APN function then its
linearity is $\leq 2^{n-1}$ (see \cite{Carlet2018}).
Here quadratic means that all component functions have degree $2$,
or, in other words, the function
$x\mapsto F(x+a)+F(x)+F(a)+F(0)$ is linear for $a\in\binvn{n}$.
Let us formulate a general upper bound:

\begin{cor} \label{cor:ApnLinUpperBound}
 If $F\colon \binvn{n} \to \binvn{n}$ is APN then $\lin{F} \leq 2^n-4$ for $n\geq 3$.
\end{cor}
\begin{proof}
Assume $n\geq 4$. Since $\abs{G_F}$ is even,
it follows that $\lin{F}$ is also even.
Assuming $\lin{F} = 2^{n}-2$ would lead to a Sidon set
in $\binvn{2n-1}$ of size $2^{n-1} + 2^{n-1} - 1 = 2^n-1$ due to \corref{cor:main}.
But this is not possible because of \propref{prop:max-sidon-upper-bound}
and the case $n\geq 4$ is shown.
For $n=3$ it is an exercise to show that only one
APN function exists up to CCZ-equivalence
(see \cite[p. 28]{CarletBook2020} for the definition of CCZ-equivalence),
and that it is quadratic with linearity 4, hence $\lin{F}=4$.
\end{proof}

The connection between Sidon sets and APN functions also shows
that no component function of an APN function can have weight $1$,
i.e. no component can have all its values 0 with only one exception.
To see this, let $F\colon \binvn{n} \to \binvn{n}$ be a function
and let $a\in\binvn{n}\setminzero$. Then the \mathKeyword{component function}
$F_a \colon \binvn{n}\to \binf$ of $F$
is defined as $x\mapsto a\cdot F(x)$.

\begin{cor}\label{cor:APNComponentFunctionWeight}
  Let $F\colon \binvn{n} \to \binvn{n}$ be an APN function,
  $a\in\binvn{n}\setminzero$,  $\lambda\in\binf$
  and define $w_{a,\lambda} = \abs{\sett{x\in\binvn{n}}{F_a(x)=\lambda}}$.
  Then there exists a Sidon set of size $2^n-w_{a,\lambda}$ in $\binvn{2n-1}$.
\end{cor}
\begin{proof}
  The set $H_a=\sett{(x,y)\in\binvn{2n}}{(0,a)\cdot(x,y) = 0}$
  defines a hyperplane in $\binvn{2n}$.
  Therefore the sets $G_F\cap H_a$ and $(G_F+b)\cap H_a$ are
  Sidon sets, where $b\in \binvn{2n}\setminus H_a$.
  If $\lambda=0$, we have $|G_F\cap H_a|=w_{a,\lambda}$,
  and if $\lambda=1$, we have $|(G_F+b)\cap H_a|=w_{a,\lambda}$,
  see also \propref{prop:WalshLinConnection}.
\end{proof}

\begin{cor}
  For $n\geq 3$ there is no APN function $F\colon \binvn{n} \to \binvn{n}$
  with component function $x_1x_2\dotsc x_n$.
\end{cor}
\begin{proof}
Assume $n\geq 4$.
Note that $x_1x_2\dotsc x_n$ gives a component function with $w_{a,1}=1$,
which results in a Sidon set of size $2^n-1$ in $\binvn{2n-1}$
due to \corref{cor:APNComponentFunctionWeight}.
But this Sidon set cannot exist because of \propref{prop:max-sidon-upper-bound}
and the case $n\geq 4$ is shown.
%The case $n=3$ is easy to calculate.}
The case $n=3$ follows again from the
uniqueness of such an APN function.
\end{proof}

\section{New construction} \label{sec:NewConstruction}

It follows from \corref{cor:main} that it is sufficient
to know the linearity of an APN function
in order to know the best possible size
when intersecting its graph with a hyperplane.
Therefore we discuss in this section the known linearities of APN functions.
There are many good research papers and books which give
a good overview about APN functions and their Walsh spectra, we refer to
\cite{BudaghyanBook, CarletBook2020, Pott2016}.

A classical result from Chabaud and Vaudenay \cite{ChabaudVaudenay1995} shows that the
linearity of a function $F\colon \binvn{n} \to \binvn{n}$ is lower bounded by $2^{(n+1)/2}$
and that can be reached for $n$ odd only.
Functions reaching this bound are called \mathKeyword{almost bent}.
They are APN and are characterised
by having only the values $0$, $2^{(n+1)/2}$ and $-2^{(n+1)/2}$ as their Walsh coefficients \cite{ChabaudVaudenay1995}.
Let us begin with quadratic APN functions $F\colon \binvn{n} \to \binvn{n}$.
Several infinite families of quadratic APN functions as well as sporadic ones are known.

If $n$ is odd, quadratic APN functions are almost bent, see \cite{CarletBook2020}.
We may apply \corref{cor:main} to these APN functions and obtain Sidon sets
of size $2^{n-1} + 2^{\frac{n-1}{2}}$ in $\binvn{2n-1}$ for $n$ odd.
This size matches the size of Sidon sets with classical parameters
(see \remref{rem:classical}).
Hence we do not obtain Sidon sets larger than the known ones
but we remark that all APN functions for $n$ odd which are not almost bent
will lead to improvements.

The situation is more interesting if $n$ is even.
The linearities of the known \highlight{infinite} families of
quadratic APN functions for $n$ even are $2^{n/2+1}$.
Applying \corref{cor:main} results in Sidon sets of size
$2^{n-1} + 2^{n/2}$ in $\binvn{2n-1}$ for $n$ even.
Again, this size matches the size of Sidon sets with classical parameters
(see \remref{rem:classical}).
Many \highlight{sporadic} quadratic APN functions for
$n=6$ and $n=8$ are known, see
\cite{BeierleLeander2020,BrowningDillonKiblerMcQuistan2009,YuPerrin2022,YuWangLi2014}.
It is interesting that some of them have the best possible linearity $2^{n-1}$
(see \secref{sec:Linearity}).
In the case $n=6$ (see \cite{BrowningDillonKiblerMcQuistan2009}) there is one
quadratic function with linearity $32$.
Unfortunately, this only yields a Sidon set of size $48$ in $\binvn{11}$,
which is not an improvement as it is equal to the best previously known example,
see \tabref{tab:SidonSizesOverview}.
In the case $n=8$ we have examples with linearity $32$ (the one with classical parameters),
%and also with linearity $64$ and $128$ (see \cite{BeierleLeander2020})
with linearity $64$ and, thanks to Beierle and Leander \cite{BeierleLeander2020},
with linearity $128$.
If we use \corref{cor:main} for a function with linearity $128$,
we obtain the following remarkable improvement
on the size of the largest known Sidon set
(which was $152$) in dimension $15$:

\begin{thm}
There exists a Sidon set in $\binvn{15}$ of size $192$.
\end{thm}

A maximal Sidon set in $\binvn{15}$ of size $192$ is given in \tabref{tab:maxSidonExampleN15}.
In the table we use the standard integer representation of vectors in $\binvn{t}$:
the integer $\sum_{i=0}^{t-1} a_i2^i$ in 2-adic representation
``is'' the vector $(a_0,\ldots, a_{t-1})$.

According to \cite{BeierleLeanderPerrin2022}, there are exactly four CCZ-inequivalent
APN functions on $8$ variables with linearity $128$
(see \cite[p. 28]{CarletBook2020} for the definition of CCZ-equivalence).
These lead to Sidon sets of size 192.

We showed that we obtain exactly one Sidon set of size 192 (up to affine permutations)
for each of the four APN functions. Moreover, these four Sidon sets are inequivalent
(which means that there is no affine permutation mapping one onto the other).
This has also been pointed out to us by Gábor Nagy.
We do not see a reason why CCZ-inequivalent APN functions must give rise
to inequivalent Sidon sets, therefore we checked the inequivalence of the
Sidon sets explicitely using the Magma algebra system \cite{Magma}:
we showed that the corresponding extended codes are inequivalent, see also
\cite{BrowningDillonKiblerMcQuistan2009}.
One of the anonymous reviewers pointed out that the inequivalence
can also be shown by calculating the exclude distribution of the Sidon sets
(see \cite{Thornburgh2024}).

\begin{table}[ht!]
  \moreSpaceBetweenRows{1.2}
  \centering
  \begin{tabular}{@{}crrrrrrrrc@{}}
\{
&    0,&     1,&     2,&     4,&     8,&    16,&    32,&    64,& \\
&  128,&   256,&   423,&   512,&   732,&   861,&  1024,&  1177,& \\
& 1228,&  1335,&  1350,&  1481,&  1594,&  1934,&  2048,&  2415,& \\
& 2719,&  2824,&  3421,&  3761,&  3839,&  3893,&  3904,&  4096,& \\
& 4233,&  4446,&  4726,&  4834,&  5013,&  5199,&  5322,&  5366,& \\
& 5410,&  5515,&  5713,&  6172,&  6405,&  6451,&  6521,&  6642,& \\
& 6672,&  6862,&  6888,&  6914,&  7358,&  7681,&  7909,&  8040,& \\
& 8192,&  8284,&  8292,&  8488,&  8700,&  8913,&  9191,&  9219,& \\
& 9398,&  9817,&  9875,& 10262,& 10378,& 11071,& 11153,& 11647,& \\
&11706,& 11752,& 11911,& 11998,& 12197,& 12295,& 12334,& 12392,& \\
&12484,& 12555,& 12646,& 12927,& 12965,& 13110,& 13682,& 13858,& \\
&14068,& 14138,& 14176,& 14236,& 14391,& 14430,& 14572,& 14579,& \\
&14808,& 15005,& 15008,& 15205,& 15238,& 15583,& 15843,& 16008,& \\
&16384,& 16409,& 16593,& 16714,& 17032,& 17491,& 17639,& 17697,& \\
&17741,& 17942,& 18157,& 18201,& 18261,& 18342,& 18423,& 18531,& \\
&18582,& 18738,& 18819,& 18941,& 19109,& 19425,& 19507,& 19666,& \\
&19728,& 19779,& 20100,& 20225,& 20384,& 20565,& 21605,& 21933,& \\
&21994,& 22077,& 22228,& 22377,& 22647,& 22650,& 22705,& 23047,& \\
&23486,& 23495,& 23582,& 23680,& 24332,& 24346,& 24443,& 24444,& \\
&24590,& 24780,& 24803,& 25048,& 25185,& 25405,& 25968,& 25997,& \\
&26016,& 26553,& 26949,& 27003,& 27092,& 27310,& 27613,& 27821,& \\
&28255,& 28550,& 28586,& 29213,& 29224,& 29264,& 29342,& 29386,& \\
&29429,& 29463,& 29688,& 29786,& 29926,& 30412,& 30663,& 30900,& \\
&31350,& 31488,& 31563,& 31768,& 31854,& 31907,& 32079,& 32365\;&
\}
  \end{tabular}

  \caption{ \label{tab:maxSidonExampleN15}
    Example of maximal Sidon set in $\binvn{15}$ of size 192.
  }
\end{table}

\begin{rem}
    No infinite family of APN functions on $\binvn{n}$
    with linearity $2^{n-1}$ is known so far,
    and there is also no proof that such a family cannot exist,
    see \cite{BeierleCarlet2023} for a thorough investigation
    of such APN functions. If APN functions on $\binvn{n}$
    with linearity $2^{n-1}$ exist for all even $n$,
    we would have Sidon sets of size
    $2^{n-1} + 2^{n-2} = 3\cdot2^{n-2}$ in $\binvn{2n-1}$.
\end{rem}

Next we discuss the inverse function $F(x)=x^{-1}$ on $\binfn{n}$ (here $F(0)=0$).
This is an APN function if and only if $n$ is odd \cite{Nyberg1993}.
The Walsh coefficients are the elements $\equiv 0 \bmod 4$
in the interval
$$
[-2^{n/2+1}+1, 2^{n/2+1}+1],
$$
see \cite{LachaudWolfmann1990}, and therefore:

\begin{thm}
  The linearity of the inverse function equals to
  $$
    \floor{2^{n/2+1}+1} - \left( \floor{2^{n/2+1}+1} \bmod 4\right)
  $$
  and leads to a Sidon set of size
  $$
    2^{n-1} + \frac{1}{2}\Bigl(\floor{2^{n/2+1}+1} - \left( \floor{2^{n/2+1}+1} \bmod 4\right)\Bigr)
  $$
  in $\binvn{2n-1}$ for odd $n\geq 3$.
\end{thm}
\begin{proof}
  The upper bound for the Walsh coefficients would give a linearity of
  \begin{equation}
      \floor{2^{n/2+1}+1} - \left( \floor{2^{n/2+1}+1} \bmod 4\right)
  \end{equation}
  and the lower bound of
  \begin{multline}
    \abs{\floor{-2^{n/2+1}+1} + 4 - \left( \floor{-2^{n/2+1}+1} \bmod 4\right)} =\\
      -\floor{-2^{n/2+1}+1} - 4 + \left( \floor{-2^{n/2+1}+1} \bmod 4\right).
  \end{multline}
  We define $\eps = 2^{n/2+1}+1 - \floor{2^{n/2+1}+1}$.
  Then
  $$
  \floor{-2^{n/2+1}+1} = -2^{n/2+1}+1 - (1-\eps) = -2^{n/2+1}+\eps.
  $$
  Hence
 \begin{equation}
    \floor{2^{n/2+1}+1} + \floor{-2^{n/2+1}+1} = 2^{n/2+1}+1 - \eps + -2^{n/2+1}+\eps = 1
 \end{equation}
  and
 \begin{equation}
    \left( \floor{2^{n/2+1}+1} \bmod 4\right) \ne \left( \floor{-2^{n/2+1}+1} \bmod 4\right).
 \end{equation}
  Subtracting (3) from (2) and using (4) gives
  $$
    5 - \left( \floor{2^{n/2+1}+1} \bmod 4\right) - \left( \floor{-2^{n/2+1}+1} \bmod 4\right)
  $$
  which is always $\geq 0$ because of (5).
  This shows the linearity of the inverse.
  The existence of the Sidon set follows from \corref{cor:main}.
\end{proof}

The inverse function is not almost bent
for odd $n\geq 5$ because of its Walsh coefficients. This results in an improvement in size
compared to the Sidon sets with classical parameters (see \remref{rem:classical}).
A few concrete values for the size of Sidon sets using the inverse function are contained in \tabref{tab:SidonSizesOverview}.

Let us finally discuss the Dobbertin functions
\cite{Dobbertin2001}.
It is an infinite family of APN functions $\binvn{n}\to\binvn{n}$
which is not almost bent \cite{CanteautCharpinDobbertin2000}.
The Walsh spectrum and the linearity is not known, but
there is a conjecture about it, see
\cite{BudaghyanCalderiniCarletDavidovaKaleyski2020}.
The authors of that paper computed the Walsh spectrum
for several cases $n\leq 35$;
in the range of \tabref{tab:SidonSizesOverview},
the cases $n=5$ and $n=10$ are important:
the linearities are $12$ ($n=5$) and $80$ ($n=10$).
For $n=15$, the linearity is $576$,
a value which is larger than the linearity of the inverse function.
But let us first note that the Dobbertin function
$\binvn{10}\to \binvn{10}$ can be used to improve the best known lower bound
on the size of Sidon sets in
dimension $t=19$, once more using \corref{cor:main}:

\begin{thm}
    There exists a Sidon set of size $552$ in $\binvn{19}$.
\end{thm}

\begin{rem}
Dobbertin functions exist for all $n$ divisible by $5$.
If the conjecture in
\cite{BudaghyanCalderiniCarletDavidovaKaleyski2020}
about the Walsh spectrum of these functions is true, then
the general formula for the linearity of the Dobbertin function is
$2^{3n/5} + 2^{2n/5}$
and leads to Sidon sets of size
$$2^{n-1} + 2^{3n/5-1} + 2^{2n/5-1}\mbox{\ in }\binvn{2n-1}.$$
For the case $n$ is divisible by $5$, when the Dobbertin function is defined, this size is larger than the size of Sidon sets with classical parameters (see \remref{rem:classical}).
\end{rem}

\section{Linear Codes}
\label{sec:linear-codes}

In this short section we want to point out explicitly the consequences of our Sidon sets to codes.
We use the well known one-to-one correspondence
between %sum-free
Sidon sets that contain zero and linear codes with minimum distance $\geq 5$
(\cite[p. 33]{MacWilliamsSloane1978}, \cite{CohenZemor1999}, \cite{CzerwinskiPott2024}), see
also \cite{CarletCharpinZinoviev1998} for the APN case.

A (binary) \mathKeyword{linear code} $\code$ of \mathKeyword{length} $m$
and \mathKeyword{dimension} $k$ is a $k$-dimensional linear subspace $\code$ of $\binvn{m}$
and $c\in \code$ is called a \mathKeyword{code word} of $\code$.
We consider all vectors to be row vectors.
If the \mathKeyword{minimum distance} of $\code$ is $d$,
that is the minimum number of non-zero entries
of all non-zero code words of $\code$, then $\code$ is called a $[m,k,d]$-code.
We call a $[m,m-t,d]$-code, a linear code of length $m$, with $t$ \mathKeyword{check bits}
and with minimum distance $d$.
Such a $[m,m-t,d]$-code $\code$ can be defined as the kernel of a matrix $\checkMat$
with $t$ rows, $m$ columns and rank $t$ which is called \mathKeyword{parity-check matrix}
of $\code$, i.e $\code = \sett{v\in\binvn{n}}{\checkMat\cdot \transpose{v} = 0}$.
Then a %sum-free
Sidon set $M$ in $\binvn{t}$ with $0\in M$, $\abs{M}\geq t+2$
and $\dim\vspan{M} = t$
(that is the dimension of the linear span of $M$)
leads to a linear code $\code_M$ of length $\abs{M}-1$,
with $t$ check bits and with minimum distance $\geq 5$.
Thus the $\abs{M}-1$ non-zero elements $M\setminzero$
form the columns of the parity-check matrix of $\code_M$.
Since every Sidon set contains, without loss of generality, the zero vector, we obtain:
{\it{}Every Sidon set $M$ in $\binvn{t}$ with $\abs{M}\geq t+2$ and $\dim\vspan{M} = t$
leads to a $[\abs{M}-1,\abs{M}-t-1,\geq 5]$-code.}

In the following theorem, we summarise the consequences of the results of the previous section
to coding theory.
\begin{thm} \label{thm:CodesImprovements}
\begin{enumerate}
  \item There exists a $[191,176,5]$-code.
  \item There exist $[m - 1, m - 2n - 2,5]$-codes for
        $$m = 2^{n-1} + \frac{1}{2}\left(\floor{2^{n/2+1}+1} - \left( \floor{2^{n/2+1}+1} \bmod 4\right)\right)$$
        and odd $n\geq 3$.
  \item There exist $[m - 1, m - 2n - 2,5]$-codes for
        $$m = 2^{n-1} + 2^{3n/5-1} + 2^{2n/5-1}$$
        and $n=5,10,15,20,25,30,35$.
\end{enumerate}
\end{thm}

The $[191,176,5]$-code from \thmref{thm:CodesImprovements} (a) improves the best previously known length 151 of a linear code with $15$ check bits and with minimum distance $5$ (see \cite{Bhargava1983,Grassl:codetables}).
The codes from \thmref{thm:CodesImprovements} (b)
provide an infinite family of codes
with better length for the same odd number of check bits $t$ and minimum distance 5
than the codes from the Sidon sets with classical parameters (see \remref{rem:classical}).

\section{Summary and Table}
\label{sec:summary}

In this work we constructed large Sidon sets in $\binvn{2n-1}$
in \secref{sec:NewConstruction}.
We obtained these Sidon sets by intersecting the graph of a highly linear APN function
$\binvn{n}\to\binvn{n}$ with affine hyperplanes.

The inverse function gives an infinite family
of Sidon sets larger than those with classical
parameters if $n$ is odd\footnote{
An infinite family of Sidon sets, which improves the size
obtained by the inverse function by
one, was recently given in \cite{Thornburgh2025}.
}.
The known linearities of the Dobbertin functions give improvements.

In the case $n=8$, sporadic examples of quadratic APN functions
with high linearity $2^{7}$ lead to Sidon sets of size 192
in $\binvn{15}$, where the largest sets so far had size $152$.
If it is possible to find an infinite family of quadratic
APN functions on $\binvn{n}$ with linearity $2^{n-1}$
we would have Sidon sets of size $3\cdot2^{n-2}$.
This would give a considerable improvement
of the best known Sidon size.

Recall that each of the new large Sidon set $M$ in $\mathbb{F}_2^t$ yields a linear code
with $t$ check bits, minimum distance 5, and a length not known so far
(see \secref{sec:linear-codes}).

\tabref{tab:SidonSizesOverview} gives an overview about
the maximum size of Sidon sets in small dimensions.
In the table, the best known upper bound and the best known size so far
for dimension $\leq 15$ come from
\cite{Grassl:codetables} (\url{https://www.codetables.de}).
%For dimension $\geq 16$ they come from
%\cite{MinT:codetables} (\url{http://mint.sbg.ac.at}).
%updated by the upper bound from \cite{CzerwinskiPott2024}
%and by the classical sizes,
%which are the sizes of the Sidon sets with classical parameters.
For larger dimensions see
\cite{MinT:codetables} (\url{http://mint.sbg.ac.at}).
However, the classical sizes,
which are the sizes of the Sidon sets with classical parameters,
and improvements on the upper bound from \cite{CzerwinskiPott2024} have not yet been added.

It is interesting that the connection between
Sidon sets and APN functions which we discussed in this paper
can also be used to improve the upper bound
for the linearity of APN functions
(\corref{cor:ApnLinUpperBound}).

\begin{table}[ht]
  \moreSpaceBetweenRows{1.2}
  \centering

  \resizebox{\textwidth}{!}{%
  \begin{tabular}{@{}rcccc cccc ccccc@{}}\toprule
    $t$ &
    3 & 4 & 5 & 6  & 7  & 8  & 9  & 10 & 11 & 12 & 13  & 14  & 15 \\
    \midrule
    Best known upper bound &
    4 & 6 & 7 &  9 & 12 & 18 & 24 & 34 & 58 & 89 & 125 & 179 & 254 \\
    Best known size so far &
    4 & 6 & 7 &  9 & 12 & 18 & 24 & 34 & 48 & 66 & 82 & 129 & 152 \\
    Classical size &
    4 & 6 & 6 & 9 & 12 & 18 & 20 & 33 & 40 & 66 & 72 & 129 & 144 \\
    \midrule
    Dobbertin &
     & &   & &    & & 22 & &    & &    & &     \\
    Inverse &
     & & 6 & &    & & 22 & &    & & 74 & &     \\
    Quadratic, highest lin. &
     & &   & & 12 & &    & & 48 & &    & & \new{192} \\
    \bottomrule
  \end{tabular}
}

\resizebox{\textwidth}{!}{%
  \begin{tabular}{@{}rcccc cccc cccc@{}}\toprule
    $t$ &
    16 & 17 & 18 & 19 & 20 & 21 & 22 & 23 & 24 & 25  \\
    \midrule
    Upper bound from \cite{CzerwinskiPott2024} &
    360 & 510 & 723 & 1022 & 1446 & 2046 & 2895 & 4094 & 5791 & 8190 \\
%    Best known size so far &
%    258 & 272 & 513 & 544 & 1026 & 1056 & 2049 & 2112 & 4098 & 4160 \\
    Classical size &
    258 & 272 & 513 & 544 & 1026 & 1056 & 2049 & 2112 & 4098 & 4160 \\
        \midrule
    Dobbertin &
      &     &  & \new{552} \\
    Inverse &
      & \new{278} & & & & \new{1068} & & & & \new{4186} \\
%    Quadratic, highest lin. &
%      &     &  &     & &      & &      & &     \\
    \bottomrule
  \end{tabular}
}
  \caption{ \label{tab:SidonSizesOverview}
    Sizes of Sidon sets in $\binvn{t}$ and related bounds and constructions.
  }
\end{table}

\section*{Acknowledgement}
  We thank Claude Carlet and Gábor Nagy for useful comments on an early version of this paper
  and we are grateful to Alexandr Polujan for fruitful discussions.
  Our thanks also go to the anonymous reviewers for their valuable comments.
  Part of the research of the second author has been funded by the
  Deutsche Forschungsgemeinschaft (DFG, German Research Foundation) – Project No. 541511634 .

\printbibliography

\end{document}